\documentclass[10pt]{amsart}

\usepackage{amsmath}
\usepackage{amsfonts}
\usepackage{amssymb}
\usepackage{amsthm}
\usepackage{url}
\usepackage{tikz-cd}
\usepackage{graphicx}
\usepackage{caption}
\usepackage{subcaption}
\usepackage{comment}
\usepackage[colorlinks,pagebackref,hypertexnames=false]{hyperref}
\usepackage{color}
\usepackage{enumerate}
\usepackage{mathrsfs}

\usepackage{todonotes}

\numberwithin{equation}{section}

\newtheorem{proposition}{Proposition}[section]
\newtheorem{lemma}[proposition]{Lemma}
\newtheorem{theorem}[proposition]{Theorem}
\newtheorem{corollary}[proposition]{Corollary}

\theoremstyle{definition}
\newtheorem{remark}[proposition]{Remark}
\newtheorem{definition}[proposition]{Definition}

\DeclareMathOperator{\End}{End}
\DeclareMathOperator{\Id}{Id}

\DeclareMathOperator{\GL}{GL}
\DeclareMathOperator{\Aut}{Aut}

\newcommand{\dbar}{\bar\partial}

\newcommand{\scGC}{\mathscr{G}^{\mathbb{C}}}

\newcommand{\scV}{\mathcal{V}}
\newcommand{\rk}{\mathrm{rank}}
\newcommand{\Vol}{\mathrm{Vol}}
\newcommand{\Hom}{\mathrm{Hom}}

\def\R{\mathbb{R}}
\def\Q{\mathbb{Q}}

\def\N{\mathbb{N}}

\def\P{\mathbb{P}}

\def\C{\mathbb{C}}

\def\cC{\mathcal{C}}

\def\cF{\mathcal{F}}
\def\cE{\mathcal{E}}
\def\cO{\mathcal{O}}

\def\cG{\mathcal{G}}

\def\cV{\mathcal{V}}

\def\cL{\mathcal{L}}

\def\del{\partial}

\def\Gr{\mathrm{Gr}}

\def\delb{\overline{\partial}}

\def\rank{\mathrm{rank}}
\def\trace{\mathrm{trace}}
\def\Lie{\mathrm{Lie}}
\def\k{\mathfrak{k}}

\def\g{\mathfrak{g}}

\def\aut{\mathfrak{aut}}

\def\codim{\mathrm{codim}}

\def\Om{\Omega}
\def\om{\omega}
\def\ep{\varepsilon}
\def\>{\rangle}

\def\<{\langle}
\def\>{\rangle}

\title[Blowing-up HYM connections]{Blowing-up hermitian Yang--Mills connections}

\author[A. Clarke]{Andrew Clarke}
\address{Andrew Clarke, Instituto de Matem\'atica, Universidade Federal do Rio de Janeiro, Av. Athos da Silveira Ramos 149, Rio de Janeiro, RJ, 21941-909, Brazil}
\email{andrew@im.ufrj.br} 

\author[C. Tipler]{Carl Tipler}
\address{Carl Tipler, Univ Brest, UMR CNRS 6205, Laboratoire de Math\'ematiques de Bretagne Atlantique, France}
\email{Carl.Tipler@univ-brest.fr}


\begin{document}

\begin{abstract}
We investigate hermitian Yang--Mills connections for pullback vector bundles on blow-ups of K\"ahler manifolds along submanifolds. Under some technical asumptions on the graded object of a simple and semi-stable vector bundle, we provide a necessary and sufficent numerical criterion for the pullback bundle to admit a sequence of hermitian Yang--Mills connections for polarisations that make the exceptional divisor sufficiently small, and show that those connections converge to the pulled back hermitian Yang-Mills connection of the graded object.
\end{abstract}

\maketitle 

\section{Introduction}
\label{sec:intro}
A cornerstone in gauge theory is the Hitchin--Kobayashi correspondence (\cite{skobayashi82,lubke83,uhlenbeckyau86,donaldson87}). This celebrated generalisation of the Narasimhan and Seshadri theorem asserts that a holomorphic vector bundle over a K\"ahler manifold carries an Hermite--Einstein metric if and only if it is polystable in the sense of Mumford and Takemoto (\cite{Mum62,Tak72}). The interplay between the differential geometric side, in the form of the hermitian Yang--Mills connections (HYM for short) that originated from physics, and the algebro-geometric side, that of stability notions coming from moduli constructions, has had many applications and become a very fertile source of inspiration. Given that HYM connections are canonically attached to polystable vector bundles, it is natural to investigate their relations to natural maps between vector bundles, such as pullbacks. In this paper, we address the problem of pulling back HYM connections along blow-ups. While the similar problem for extremal K\"ahler metrics has seen many developments in the past ten years \cite{ArPac06,ArPac09,ArPacSing,sze-blowupsII,SeySze,dervan-sektnan-blowups}, relatively little seems to be known about the behaviour of HYM connections under blow-ups \cite{Buchdahl,dervansektnan19}. In this paper, under some mild asumptions, we solve the problem for pullback of {\it semi-stable} vector bundles on blow-ups along smooth centers.

Let $\pi: X' \to X$ be the blow-up of a polarised K\"ahler manifold $(X,[\om])$ along a submanifold $Z \subset X$, and $E'=\pi^*E$ the pullback of a holomorphic vector bundle $E\to X$. For $0<\ep\ll 1$, $L_\ep:=\pi^*[\om] -\ep c_1(Z')$ defines a polarisation on $X'$, where we denote by $Z'=\pi^{-1}(Z)$ the exceptional divisor. There are obstructions for $E'$ to admit HYM connections with respect to $\om_\ep\in c_1(L_\ep)$, with $0<\ep\ll 1$. In particular, $E$ should be {\it simple} and {\it semi-stable} with respect to $[\om]$
. In the latter case, $E$ admits a Jordan--Holder filtration by semi-stable sheaves with polystable graded object $\Gr(E)$ (see Section \ref{sec:slopestability} for definitions). A further obstruction comes then from subsheaves of $E$ arising from $\Gr(E)$. While those sheaves have the same slope as $E$, their pullbacks to $X'$ could destabilise $E'$. Our main result asserts that, under some technical asumptions on $\Gr(E)$ and on the dimension of $Z$, these are actually the only obstructions for $E'$ to carry a HYM connection. More precisely, from now on, and {\it until the end of the paper, we will assume that $Z\subset X$ satisfies}
$$
\codim(Z)\geq 3.
$$

Recall that a semi-stable holomorphic vector bundle $E\to (X,[\om])$ is said to be {\it sufficiently smooth} if its graded object $\Gr(E)$ is locally free. We will assume further that the Jordan--H\"older filtration with locally-free subsheaves
$$
0=\cF_0\subset \cF_1\subset \ldots\subset \cF_\ell=E 
$$
is unique. For $1\leq i \leq \ell$, denote by $\mu_{L_\ep}(\cF_i)=\frac{c_1(\pi^*\cF_i)\cdot L_\ep^{n-1}}{\rank(\cF_i)}$ the slope of $\pi^*\cF_i$ on $(X', L_\ep)$. In the following statements, when referring to a HYM connection $A$ on $E$ (resp. on $\Gr(E)$, or $\pi^*E$), we will implicitly assume that $A^{0,1}$ is gauge equivalent to the holomorphic connection, or Dolbeault operator, of $E$ (resp. of $\Gr(E)$, or $\pi^*E$).
\begin{theorem}
 \label{theo:intro}
 Let $E\to X$ be a sufficiently smooth semi-stable holomorphic vector bundle on $(X,[\om])$.  Assume that the stable components of $\Gr(E)$ are pairwise non-isomorphic and that $E$ admits a unique Jordan--Holder filtration by locally free subsheaves.
  Then, there exists $\ep_0>0$ and a sequence of HYM connections $(A_\ep)_{\ep\in (0,\ep_0)}$ on $\pi^*E$ with respect to $(\om_\ep)_{\ep\in (0,\ep_0)}$ if and only if
 \begin{equation}
 \label{eq:necessarynumericalcondition}
 \forall\;  i\in [\![1,\ell-1]\!],\: \mu_{L_\ep}(\cF_i)\underset{\ep\to 0}{<} \mu_{L_\ep}(E).
\end{equation}
 In that case, if $A$ denotes a HYM connection on $\Gr(E)$ with respect to $\om$, then $(A_\ep)_{\ep\in(0,\ep_0)}$ can be chosen so that $A_\ep \underset{\ep\to 0}{\longrightarrow} \pi^*A$ in any Sobolev norm.
\end{theorem}
 The expression $\mu_{L_\ep}(\cF)\underset{\ep\to 0}{<} \mu_{L_\ep}(E)$ means that the first non-zero term in the $\ep$-expansion for $\mu_{L_\ep}(E)-\mu_{L_\ep}(\cF)$ is strictly positive. The conclusion of Theorem \ref{theo:intro} can be rephrased as follows. If $\delb_{\Gr(E)}$ (resp. $\delb_E$) stands for the Dolbeault operator of $\Gr(E)$ (resp. of $E$) and if $h$ is a Hermite--Einstein metric on $\Gr(E)$ with respect to $\om$, then there exist gauge transformations $(f_\ep)_{0<\ep<\ep_0}$ on $X'$ such that the Chern connections of $(f_\ep^*\pi^*\delb_E,\pi^*h)$ are HYM with respect to $(\om_\ep)_{0<\ep<\ep_0}$ and converge to the Chern connection of $(\pi^*\delb_{\Gr(E)},\pi^*h)$ in any Sobolev norm.
 \begin{remark}
  The simplicity of $E$ is implied by the hypotheses made on $\Gr(E)$ (see \cite[Lemma 34]{SekTip}). Semi-stability and condition (\ref{eq:necessarynumericalcondition}) are also necessary to produce the connections $(A_\ep)$ from Theorem \ref{theo:intro}. The other three asumptions on $\Gr(E)$ are technical. Assuming $\Gr(E)$ to be locally free enables one to regard $E$ as a smooth complex deformation of $\Gr(E)$ and to work with the various connections on the same underlying complex vector bundle. We should warn the reader though that if one drops this asumption, Condition (\ref{eq:necessarynumericalcondition}) may not be enough to ensure semi-stability of $\pi^*E$ on $(X',L_\ep)$ (see the extra conditions in \cite[Theorem 1.10]{NapTip}). On the other hand, the asumption on $\Gr(E)$ having no pairwise isomorphic components is purely technical, and ensures that its automorphism group, which will provide obstructions in the perturbative theory, is abelian. Finally, the fact that the Jordan--Holder filtration is unique provides a special shape in the extension matrix from $\Gr(E)$ to $E$, with non-zero terms above the diagonal. This is used to produce approximate solutions by induction on the number of stable components of $\Gr(E)$.
 \end{remark}
We now list some corollaries of Theorem \ref{theo:intro}. First, the stable case :
\begin{corollary}
 \label{cor:stablecase}
 Let $E\to X$ be a stable holomorphic vector bundle on $(X,[\om])$ and let $A$ be a HYM connection on $E$ with respect to $\om$. Then, there exists $\ep_0>0$ and a sequence of HYM connections $(A_\ep)_{\ep\in (0,\ep_0)}$ on $\pi^*E$ with respect to $(\om_\ep)_{\ep\in (0,\ep_0)}$ such that $A_\ep \underset{\ep\to 0}{\longrightarrow} \pi^*A$ in any Sobolev norm.
\end{corollary}
For the semi-stable case, Condition (\ref{eq:necessarynumericalcondition}) reduces to a finite number of intersection product computations. One interesting feature comes from the second term in the expansion of $\mu_{L_\ep}(E)$. It is the opposite of the slope of the restriction of $E$ to $Z$. The following formula is proved in \cite[Section 4.1]{NapTip},  where $m=\dim(Z)$ :
\begin{equation}
 \label{eq:formulaAchim}
\mu_{L_\ep}(E)= \mu_{L}(E) - \dbinom{n-1}{m-1} \mu_{L_{\vert Z}}(E_{|Z}) \ep^{n-m} + O(\ep^{n-m+1}).
\end{equation}
We then have :
\begin{corollary}
 \label{cor:restrictionnumericalcriterion}
 Let $E\to X$ be a sufficiently smooth semi-stable holomorphic vector bundle on $(X,[\om])$. Assume that the stable components of $\Gr(E)$ are pairwise non-isomorphic and that the Jordan--Holder filtration with locally free subsheaves is unique. Denote by $A$ an HYM connection on $\Gr(E)$ with respect to $\om$. If
 \begin{equation}
 \label{eq:necessarynumericalconditionbis}
 \forall\; i\in [\![ 1,\ell-1]\!],\: \mu_{L_{\vert Z}}(E_{\vert Z})<\mu_{L_{\vert Z}}((\cF_i)_{\vert Z}),
\end{equation}
  then, there exists $\ep_0>0$ and a sequence of HYM connections $(A_\ep)_{\ep\in (0,\ep_0)}$ on $\pi^*E$ with respect to $(\om_\ep)_{\ep\in (0,\ep_0)}$ converging to $\pi^*A$ in any Sobolev norm.
\end{corollary}
Condition (\ref{eq:necessarynumericalconditionbis}) was checked on explicit examples in \cite[Section 4.5]{NapTip} to produce stable perturbations of tangent sheaves by blow-ups, and our result provides information on the associated connections and their asymptotic behaviour. Note that by the Mehta--Ramanathan theorem \cite{MeRa}, if $[\om]=c_1(L)$ is integral, and if $Z$ is a generic intersection of divisors in linear systems $\vert L^k \vert$, then $E_{\vert Z}$ is semi-stable as soon as $E$ is. In that case, Condition (\ref{eq:necessarynumericalconditionbis}) cannot be satisfied, and it seems unlikely that Condition (\ref{eq:necessarynumericalcondition}) will hold true. Hence, blowing-up such subvarieties tend to destabilise a semi-stable bundle.

In general, we expect that it should not be too hard to obtain stability of sufficiently smooth pulled back bundles under condition (\ref{eq:necessarynumericalcondition}) with purely algebraic methods. However, we emphasize that the Hitchin--Kobayashi correspondence doesn't provide any information on the asymptotic behaviour of the associated HYM connections, which is then the main content of Theorem \ref{theo:intro}. Nevertheless, we state the following corollary, that extends \cite[Theorem 1.10]{NapTip} to a non-equivariant situation:

\begin{corollary}
 \label{cor:semistablecase}
 Let $E\to X$ be a sufficiently smooth semi-stable holomorphic vector bundle on $(X,[\om])$.  Assume that the stable components of $\Gr(E)$ are pairwise non-isomorphic and that the Jordan--Holder filtration with locally free subsheaves is unique.
  Then, there exists $\ep_0>0$ such that $\pi^*E\to (X',L_\ep)$ is 
  \begin{enumerate}
   \item[(i)] stable if and only if for all $i\in [\![ 1,\ell-1]\!],\: \mu_{L_\ep}(\cF_i)\underset{\ep\to 0}{<} \mu_{L_\ep}(E),$
   \item[(ii)] semi-stable if and only if for all $i\in [\![ 1,\ell-1]\!],\: \mu_{L_\ep}(\cF_i)\underset{\ep\to 0}{\leq} \mu_{L_\ep}(E),$
   \item[(iii)] unstable otherwise.
  \end{enumerate}
\end{corollary}

Finally, we comment on previous related works. Theorem \ref{theo:intro} extends results from \cite{Buchdahl,dervansektnan19} where blow-ups of HYM connections along points are considered. In the present paper, we consider blow-ups along any smooth subvariety, and also cover the semi-stable situation, which is technically more involved due to the presence of automorphisms of the graded object that obstruct the linear theory. While  \cite{dervansektnan19} is a gluing construction of a similar vein to the various solutions to the analogous problem of producing extremal K\"ahler metrics on blow-ups \cite{ArPac09,ArPacSing,sze-blowupsII,SeySze,dervan-sektnan-blowups}, one of the key feature in our approach is to directly apply  the quantitative implicit function theorem, closely following the method developed in \cite{SekTip}. The main new technical input is in Section \ref{sec:blowup}, where we perform a precise study of the linear operators involved in relation to the geometry of the blow-up. The rest of the argument, in Sections \ref{sec:approximate} and \ref{sec:perturbation}, is more standard, and follows the work of Sektnan and the second author \cite{SekTip}.

\subsection*{Outline:} In Section \ref{sec:setup}, we recall basic material about HYM connections and stability. We then perform in Section \ref{sec:blowup} the analysis of the linear theory on the blow-up. Relying on this, in Section \ref{sec:approximate} we explain how to produce approximate solutions to the HYM equations. This section, and the one that follows, rely on the treatment in \cite{SekTip}. Then, we perturb those approximate solutions in Section \ref{sec:perturbation} to actual solutions, which concludes the proof of Theorem \ref{theo:intro}. The corollaries are addressed in Section \ref{sec:corol}.

\subsection*{Acknowledgments:} The authors benefited from visits to LMBA and Gotheborg University; they would like to thank these welcoming institutions for providing stimulating work environments. The idea of this project emerged from discussions with Lars Martin Sektnan, whom we thank for sharing his ideas and insight. AC is partially supported by the grants {BRIDGES ANR--FAPESP ANR-21-CE40-0017} and { Projeto CAPES - PrInt UFRJ 88887.311615/2018-00}. CT is partially supported by the grants MARGE ANR-21-CE40-0011 and BRIDGES ANR--FAPESP ANR-21-CE40-0017.

\section{Preliminaries}
\label{sec:setup}
In Sections \ref{sec:HYM} and \ref{sec:slopestability} we introduce the notions of HYM connections and slope stability, together with some general results, and refer the reader to \cite{skobayashi87} and \cite{HuLe}. From Section \ref{sec:blowup} we start to specialise the discussion to blow-ups. In particular, in Section \ref{sec:decomposeLaplace}, we provide various asymptotic expressions for the linearisation of the HYM equation on the blow-up. Those results will be used in Section \ref{sec:perturbation}.

\subsection{The hermitian Yang--Mills equation}
\label{sec:HYM}

Let $E \to X$ be a holomorphic vector bundle over a compact K\"ahler manifold $X$. A hermitian metric on $E$ is \textit{Hermite--Einstein} with respect to a K\"ahler metric with K\"ahler form $\omega$ if the curvature $F_h \in \Omega^2 \left(X, \End E \right)$ of the corresponding Chern connection satisfies
\begin{align}\label{HEeqn} \Lambda_{\om} \left( iF_h \right) = c \Id_E
\end{align}
for some real constant $c$. Equivalently, if $h$ is some hermitian metric on the smooth complex vector bundle underlying $E$, a hermitian connection $A$ on $(E,h)$ is said to be \textit{hermitian Yang--Mills} if it satisfies
\begin{equation*}
\left\{
\begin{array}{ccc}
F_A^{0,2} & = & 0, \\
\Lambda_{\om} \left( i F_A \right) & = & c \Id_E.
\end{array}
\right.
\end{equation*}
The first equation of this system implies that the $(0,1)$-part of $A$ determines a holomorphic structure on $E$, while the second that $h$ is Hermite--Einstein for this holomorphic structure. We will try to find hermitian Yang--Mills connections within the complex gauge group orbit, which we now define. The \textit{complex gauge group} is
$$ 
\scGC(E) = \Gamma \left( \GL \left(E, \mathbb{C} \right) \right). 
$$
We note that the set of unitary gauge transformations preserves the metric $h$ so, taking account of the fibre-wise polar decomposition of an element of $\scGC(E)$, we consider the bundle $\End_H(E,h)$ of Hermitian endomorphisms of $E$ and the set
\begin{eqnarray*}
\mathcal{G}^\mathbb{C}(E,h):=\mathcal{G}^\mathbb{C}(E)\cap \Gamma(\End_H(E,h))=\{e^s\ :\ s\in\Gamma(\End_H(E,h))\}.
\end{eqnarray*}
If $\dbar$ is the Dolbeault operator defining the holomorphic structure on $E$, then $ f \circ \dbar \circ f^{-1}$ defines a biholomorphic complex structure on $E$. Let $d_A = \partial_A + \dbar_A$ be the Chern connection of $(E,h)$ with respect to the original complex structure (that is $\dbar_A = \dbar$). Then the Chern connection $A^f$ of $h$ with respect to $f \circ \dbar \circ f^{-1}$ is 
$$
d_{A^f} = (f^*)^{-1} \circ \partial_A \circ (f^*) + f \circ \dbar \circ f^{-1}.
$$
Solving the hermitian Yang--Mills equation is equivalent to solving $$\Psi (s) = c \Id_E$$
where 
$$
\begin{array}{cccc}
\Psi :  & \Gamma(\End_H(E,h)) &  \longrightarrow  & \Gamma(\End_H(E,h)),\\
 & s & \longmapsto & i\Lambda_\om(F_{A^{\exp(s)}}),
\end{array}
$$
and where $\Gamma(\End_H(E,h))$ is the tangent space to $\scGC(E,h)$ at the identity. For a connection $A$ on $E$, the Laplace operator $\Delta_{A}$ is
\begin{align}\label{laplaceop}  \Delta_A =  i \Lambda_{\omega} \left( \bar \partial_{A}  \partial_{A} -  \partial_{A} \bar \partial_{A} \right).
\end{align}
If $A_{\End E}$ denote the connection induced by $A$ on $\End E$, then :
\begin{lemma}\label{lem:linop} 
If $A$ is the Chern connection of $(E,\delb,h)$, the differential of $\Psi$  at identity is $$ d \Psi_{\Id_E} = \Delta_{A_{\End E}}.$$
 If moreover $A$ is assumed to be hermitian Yang--Mills, then the kernel of $ \Delta_{A_{\End E}}$ acting on $\Gamma(\End(E))$ is given by the Lie algebra $\aut(E)$ of the space of automorphisms $\Aut(E)$ of $(E,\delb)$.
\end{lemma}
The last statement about the kernel follows from the K\"ahler identities and the Akizuki-Nakano identity that imply $\Delta_{A_{\End E}}=\partial^*_{A}\partial_{A}+\bar\partial_{A}^*\bar\partial_{A}$, the two terms of which are equal if $A$ is Hermitian Yang-Mills. The operator $\Delta_{A_{\End E}}$ being elliptic and self-adjoint, $\aut(E)$ will then appear as a cokernel in the linear theory for perturbations of hermitian Yang--Mills connections.

\subsection{Slope stability}
\label{sec:slopestability}
We recall some basic facts about slope stability, as introduced by \cite{Mum62,Tak72}, and refer the interested reader to \cite{HuLe} for a detailed treatment. We denote here $L:=[\om]$ the polarisation of the $n$-dimensional K\"ahler manifold $X$.
\begin{definition}
\label{def:stability}
For $\cE$ a torsion-free coherent sheaf on $X$, the slope $\mu_L(\cE)\in\Q$ (with respect to $L$) is given by the intersection formula
 \begin{equation}
  \label{eq:slope}
  \mu_L(\cE)=\frac{\deg_L(\cE)}{\rank(\cE)},
 \end{equation}
 where $\rank(\cE)$ denotes the rank of $\cE$ while $\deg_L(\cE)=c_1(\cE)\cdot L^{n-1}$ stands for its degree. Then, $\mathcal{E}$ is said to be \emph{slope semi-stable} (resp. \emph{slope stable}) with respect to $L$ if for any coherent subsheaf $\mathcal{F}$ of $\mathcal{E}$ with $0<\rank( \cF)<\rank(\cE)$, one has
\begin{eqnarray*}
\mu_L(\mathcal{F})\leq \mu_L(\mathcal{E}) \:\textrm{( resp. } \mu_L(\mathcal{F})< \mu_L(\mathcal{E}) ).
\end{eqnarray*}
A direct sum of slope stable sheaves of the same slope
 is said to be \emph{slope polystable}.
\end{definition}
In this paper, we will often omit ``slope'' and simply refer to stability of a sheaf, the polarisation being implicit. We will make the standard identification of a holomorphic vector bundle $E$ with its sheaf of sections, and thus talk about slope stability notions for vector bundles as well. In that case slope stability relates nicely to differential geometry via the Hitchin--Kobayashi correspondence :
\begin{theorem}[\cite{skobayashi82,lubke83,uhlenbeckyau86,donaldson87}]
\label{thm:HKcorrespondence}
There exists a Hermite--Einstein metric on $E$ with respect to $\omega$ if and only if $E$ is polystable with respect to $L$
\end{theorem}
We will be mostly interested in semi-stable vector bundles. A \emph{Jordan--H\"older filtration} for a torsion-free sheaf $\cE$ is a filtration by coherent saturated subsheaves:
 \begin{align}
\label{eq:JHfiltration}
0 = \cF_0 \subset \cF_1 \subset \hdots \subset \cF_\ell = \cE,
\end{align}
such that the corresponding quotients,
\begin{align}
\label{eq:JHquots}
\cG_i = \frac{\cF_i}{\cF_{i-1}},
\end{align}
 for $i=1, \hdots, \ell$, are stable with slope $\mu_L(\cG_i)=\mu_L(\cE)$.
 In particular, the graded object of this filtration
 \begin{equation}
  \label{eq:JHgraded}
  \Gr(\cE):=\bigoplus_{i=1}^l \cG_i
 \end{equation}
is polystable. From \cite[Section 1]{HuLe}, we have the standard existence and uniqueness result:
\begin{proposition}
 \label{prop:JHfiltration}
 Any semi-stable coherent torsion-free sheaf $\cE$ on $(X,L)$ admits a Jordan--H\"older filtration, and the double dual of the graded object $\Gr(\cE)^{**}$ of such filtrations is unique up to isomorphism.
\end{proposition}
When $E$ is locally-free and semi-stable, we say that it is {\it sufficiently smooth} if $\Gr(E)$ is locally-free. In that case, we denote $\mathfrak{E}_{[\omega]}$ the set of holomorphic subbundles of $E$ built out of successive extensions of some of the stable components of $\Gr(E)$. Equivalently, $\mathfrak{E}_{[\omega]}$ is the set of holomorphic subbundles of $E$ arising in a Jordan-Holder filtration for $E$. Finally, we recall that a necessary condition for $E$ to be stable is simplicity, that is $\Aut(E)=\C^*\cdot\Id_E$.

\section{Geometry of the blow-up}
\label{sec:blowup}
We consider now $Z \subset X$ a $m$-dimensional complex submanifold of codimension $r=n-m\geq 2$ and the blow-up map 
$$
\pi : \mathrm{Bl}_Z(X) \to X.
$$
We will denote by $X'=\mathrm{Bl}_Z(X)$ the blown-up manifold and by $Z'=\pi^{-1}(Z)$ the exceptional divisor.  We denote by 
$$
L_\ep:= \pi^* L - \ep [Z']
$$
a polarisation on $X'$, for $0< \ep \ll 1$. Let $E\to X$ be a holomorphic vector bundle, and denote by $E'=\pi^*E$ the pulled back bundle. For any holomorphic subbundle $F\subset E$, the intersection numbers $\mu_{L_\ep}(\pi^*E)-\mu_{L_\ep}(\pi^*F)$ admit expansions in $\ep$, with first term given by $\mu_L(E)-\mu_L(F)$. For that reason, given the Hitchin--Kobayashi correspondence in Theorem \ref{thm:HKcorrespondence}, semi-stability of $E$ on $(X,L)$ is a necessary condition for its pullback $E'$ to admit an HYM connection with respect to a K\"ahler metric in $L_\ep$, for all $0< \ep \ll 1$. Another necessary condition is simplicity of $E'$, which, by Hartogs' theorem, is equivalent to simplicity of $E$. Then, natural candidates to test for stability of $E'$ are given by the pullbacks of elements in $\mathfrak{E}_{[\omega]}$, and Condition (\ref{eq:necessarynumericalcondition}) clearly is necessary for $E'$ to be stable in the polarisations we consider, and thus to admit an HYM connection. Hence, we will assume $E$ to be simple, semi-stable, and to satisfy (\ref{eq:necessarynumericalcondition}). We now turn back to the differential geometry of the blow-up.

\subsection{Decomposition on spaces of sections}
\label{sec:decompositions}
We have a commutative diagramm:
$$
\begin{array}{ccc}
 Z' & \stackrel{\iota}{\longrightarrow} & X' \\
 \downarrow &  & \downarrow \\
 Z & \stackrel{\iota_0}{\longrightarrow} & X
\end{array}
$$
where $\iota_0$ and $\iota$ denote the inclusions, while the vertical arrows are given by the projection map $\pi$. We then have a pullback map on sections 
$$
\pi^* : \Gamma(X, \End(E)) \longrightarrow \Gamma(X',\End(\pi^*E))
$$
as well as a restriction map :
$$
\iota^* : \Gamma(X',\End(\pi^*E)) \longrightarrow \Gamma(Z',\End(\iota^*\pi^*E)).
$$
Our goal now is to fit those maps in a short exact sequence, that will in the end split the space $\Gamma(X',\End(\pi^*E))$.
  If $N_Z=TX_{\vert Z}/ TZ$ denotes the normal bundle of $Z$ in $X$, then $Z'\simeq \P(N_Z)$, and we can fix a $(1,1)$-form $\lambda \in c_1(\cO_{\P(N_Z)}(1))$ that restricts to K\"ahler metrics on the fibers of $\P(N_Z)\to Z$. We also fix a K\"ahler form $\om\in c_1(L)$ on $X$, and consider its restriction to $Z$. We then have a K\"ahler $\C\P^{r-1}$-fibration :
$$
\pi : (Z', \lambda) \longrightarrow (Z, \om).
$$
By averaging along fibers as described in \cite[Section 2.3]{SekTip}, we obtain a splitting
\begin{equation}
 \label{eq:firstsplitting}
\Gamma(Z',\End(\iota^*\pi^*E))=\pi^*(\Gamma(Z,\End(\iota_0^*E)))\oplus\Gamma_0(Z',\End(\iota^*\pi^*E)).
\end{equation}
 We will omit the $\iota^*$ and $\pi^*$ to simplify notation. Using the projection on the second factor 
 $$
 p_0: \Gamma(Z',\End(E))\to \Gamma_0(Z',\End(E))
 $$
 in (\ref{eq:firstsplitting}), we deduce a short exact sequence :
$$
0 \longrightarrow \Gamma(X, \End(E)) \stackrel{\pi^*}{\longrightarrow} \Gamma(X',\End(E)) \stackrel{p_0\circ \iota^*}{\longrightarrow} \Gamma_0(Z',\End(E)) \longrightarrow 0.
$$
We can actually split this sequence by mean of a linear extension operator
$$
\iota_* : \Gamma_0(Z',\End(E)) \longrightarrow \Gamma(X',\End(E))
$$
such that 
$$
p_0\circ \iota^*\circ \iota_* = \mathrm{Id}.
$$
This can be done using bump functions and a standard partition of unity argument. The outcome is an isomorphism :
\begin{equation}
 \label{eq:goodcoordinates}
\begin{array}{ccc}
\Gamma(X',\End(E))  & \longrightarrow &  \Gamma(X, \End(E)) \oplus \Gamma_0(Z',\End(E)) \\
   s & \longmapsto & (s-\iota_*\circ p_0\circ \iota^* s \;,\;  p_0\circ \iota^* s),
\end{array}
\end{equation}
with inverse map $(s_X, s_Z) \mapsto (\pi^*s_X + \iota_* s_Z)$. This splits the Lie algebra of gauge transformations, and will be used to identify contributions coming from $X$ and from $Z'$ in the $\ep$-expansion of the linearisation, which we describe in the next section. From now on, by abuse of notations, we will consider the spaces $\Gamma(X, \End(E))$ and $\Gamma_0(Z',\End(E))$ as subspaces of $\Gamma(X',\End(\pi^*E))$, and denote $s=s_X+s_Z$ the decomposition of an element $s\in \Gamma(X',\End(E))$.

\subsection{Decomposition of the Laplace operator}
\label{sec:decomposeLaplace}
 We extend $\lambda$ to a closed $(1,1)$-form over $X'$ as in \cite[Section 3.3]{voisin1} and consider the family of K\"ahler metrics on $X^\prime$:
 $$\omega_\varepsilon=\pi^*\omega+\varepsilon\lambda\in c_1(L_\ep),\: 0 <\varepsilon\ll 1.$$
 Let $A$ be a Hermitian connection on $E$,
  which we pull back to $X^\prime$ and extend to the bundle $\End(\pi^*E)$. We will now study the Laplace operator 
\begin{eqnarray*}
\Delta_\varepsilon s =i\Lambda_\varepsilon (\bar\partial_{A}\partial_{A}-\partial_{A}\bar\partial_{A})s
\end{eqnarray*}
acting on the various components of $s=s_X+s_Z \in\Gamma(X^\prime,\End(E))$, where $\Lambda_\varepsilon$ is the Lefschetz operator for the metric $\omega_\varepsilon$. For this, we need to introduce an operator on $Z'$.
\begin{definition}
  The \textit{vertical Laplace operator}, denoted 
$$
\Delta_{\scV} : \Gamma_0 \left( Z', \End(E) \right) \to \Gamma_0 \left( Z',\End(E)\right),
$$ 
is the operator 
$$
\Delta_{\scV} = i \Lambda_{\cV} \big(  \partial_{F} \bar \partial_{F}  - \bar \partial_{F}  \partial_{F} \big),
$$ 
where 
$$
 \Lambda_{\cV} : \Om^{1,1}\left( Z', \End(E) \right) \to \Gamma_0 \left( Z',\End(E)\right)
$$
is the {\it vertical contraction operator} defined on each fiber $F_z=\pi^{-1}(z)$ of $\pi : Z' \to Z$  by
$$(\Lambda_{\cV} \alpha)_{\vert F_z}\, \lambda_{\vert F_z}^{r-1}= (r-1)\alpha_{\vert F_z}\wedge\lambda_{\vert F_z}^{r-2} $$
and where
$d_F=\partial_{F} +\bar \partial_{F} $ is the flat connection along the fibres of $\pi : Z' \to Z$. 
\end{definition}

%
The following lemma relies on \cite{kodairaspencer60} and was already observed in \cite[Lemma 4.6]{SekTip}. 
\begin{lemma}
 \label{lem:verticallaplacianinvertible}
 The vertical Laplacian
 $$
\Delta_{\scV} : \Gamma_0 \left( Z', \End(E) \right) \to \Gamma_0 \left( Z',\End(E)\right)
$$ 
is an invertible linear map.
\end{lemma}
 In the following statements, an expression of the form $O(\ep^{j})$ is to be understood as holding pointwise. Convergence considerations of those expressions with respect to various Sobolev space norms will be addressed in Section \ref{sec:perturbation}.
\begin{lemma}
\label{lem:exp1}
If $s_Z=\iota_*\sigma_Z$ for $\sigma_Z\in\Gamma(Z^\prime,\End(E))$, then
\begin{eqnarray*}
(p_0\circ \iota^*)\Delta_\varepsilon(\iota_*\sigma_Z)=\varepsilon^{-1}\Delta_\mathcal{V}\sigma_Z +\mathcal{O}(1).
\end{eqnarray*}
\end{lemma}
\begin{proof}
We introduce the operator $D$ given by 
\begin{eqnarray*}
Ds_Z=i(\bar\partial_{A}\partial_{A}-\partial_{A}\bar\partial_{A})s_Z.
\end{eqnarray*}
The Laplacian $\Delta_\varepsilon$ satisfies on $X^\prime$ :
$$\Delta_\varepsilon s_Z\,\omega_\varepsilon^n=n Ds_Z\wedge \omega_\varepsilon^{n-1},$$
or equivalently
\begin{eqnarray*}
\Delta_\varepsilon s_Z = \frac{n\ Ds_Z\wedge (\omega+\varepsilon\lambda)^{n-1}}{(\omega+\varepsilon\lambda)^n}.
\end{eqnarray*}
We note that $\omega$ is a K\"ahler form on $X$, but on $X^\prime$ is degenerate along the fibre directions of the submanifold $Z'$. Then $(i^*\omega)^{m+1}=0\in\Omega^{2(m+1)}(Z^\prime)$, and at $x\in Z^\prime \subseteq X^\prime$, $\omega^{m+2}=0$. 
Then, expanding $(\omega+\varepsilon\lambda)^{n-1}$ and $(\omega+\varepsilon\lambda)^n$ gives
\begin{eqnarray*} 
\iota^*\Delta_\varepsilon s_Z
= (n-m-1)\varepsilon^{-1} \frac{Ds_Z\wedge \omega^{m+1}\wedge\lambda^{n-m-2}}{\omega^{m+1}\wedge\lambda^{n-m-1}} +\mathcal{O}(1).
\end{eqnarray*}
Restricting to $Z^\prime$, the connection $1$-forms of $A$ vanish, so $\iota^*Ds_Z=i\partial\bar\partial \sigma_Z$, acting on the coefficient functions of $\sigma_Z$. On the other hand, by considering a convenient orthonormal frame at $x\in Z^\prime$, we see that $\iota^*\Delta_\varepsilon\iota_*\sigma_Z=
\varepsilon^{-1}\Delta_\mathcal{V}\sigma_Z+\mathcal{O}(1)$. 
\end{proof}
In the next lemma, we denote $\Delta_\ep s_Z=(\Delta_\varepsilon s_Z)_X+(\Delta_\varepsilon s_Z)_Z$ the decomposition according to (\ref{eq:goodcoordinates}).
\begin{lemma}
\label{lem:exp2}
For $s_Z=\iota_*\sigma_Z$ with $\sigma_Z\in\Gamma(Z^\prime,\End(E))$, we have
\begin{eqnarray*}
(\Delta_\varepsilon s_Z)_X=\mathcal{O}(1).
\end{eqnarray*}
\end{lemma}
\begin{proof}
 By definition, $(\Delta_\varepsilon s_Z)_X=\pi^*\phi$ for some $\phi\in\Gamma(X,\End(E))$. As we also have
\begin{eqnarray*}
 (\Delta_\varepsilon s_Z)_X &=& (\mathrm{Id}-\iota_*(p_0\circ\iota^*))\Lambda_\varepsilon D s_Z,
\end{eqnarray*}
we deduce that the section $\phi$ is the continuous extension of $\pi_*(\mathrm{Id}-\iota_*(p_0\circ\iota^*))\Lambda_\varepsilon D s_Z$ across $Z\subseteq X$. On $X^\prime\setminus Z^\prime$ we have
 \begin{eqnarray*}
\Lambda_\varepsilon Ds_Z &=& n\frac{Ds_Z\wedge (\omega^{n-1}+\mathcal{O}(\varepsilon))}{\omega^n+\mathcal{O}(\varepsilon)}=\cO(1).
\end{eqnarray*}
  As $\pi_*(\mathrm{Id}-\iota_*(p_0\circ\iota^*))$ is $\cO(1)$, the result follows.
\end{proof}

From the previous two lemmas, in the decomposition $$s=s_X+s_Z,$$
$\Delta_\varepsilon s_Z$ also lies in the subspace $\Gamma_0(Z^\prime,\End(E))\subseteq\Gamma(X^\prime,\End(E))$ to higher order in $\ep$. For $s_X\in\Gamma(X,\End(E))$, 
$$\Delta_\varepsilon s_X=(\Delta_\varepsilon s_X)_X+(\Delta_\varepsilon s_X)_Z$$
where $(\Delta_\varepsilon s_X)_Z=\iota_*(p_0\circ\iota^*)\Delta_\varepsilon s_X$. We first consider $\iota^*\Delta_\varepsilon s_X$.

\begin{lemma}
\label{lem:exp3}
For $s_X=\pi^*\sigma_X\in\Gamma(X,\End(E))\subseteq\Gamma(X^\prime,\End(E))$, 
\begin{eqnarray*}
\iota^*\Delta_\varepsilon s_X=(m+1)\frac{Ds_X\wedge\omega^m\wedge\lambda^{n-m-1}}{\omega^{m+1}\wedge\lambda^{n-m-1}}+\mathcal{O}(\varepsilon).
\end{eqnarray*}
\end{lemma}
\begin{proof}
Firstly, $s_X=\pi^*\sigma_X$, and the connection $A$ is pulled back from $X$, so $Ds_X$ is basic for the projection to $X$ and $Ds\wedge \omega^{m+1}=0$ at points in $Z^\prime$. Secondly, we note that $\omega^{m+1}\wedge\lambda^{n-m-1}$ is a volume form on $X^\prime$, in a neighbourhood of $Z^\prime$. Then, the result follows similarly to the previous lemma.
\end{proof}
For the final term $(\Delta_\varepsilon s_X)_X$, we introduce $\Delta_X$ the Laplace operator of $A$ on $\End (E) \to (X,\om)$:
$$
\begin{array}{cccc}
\Delta_{X} : &  \Gamma \left( X, \End(E) \right) &\to& \Gamma \left(X,\End(E)\right)\\
            & \sigma & \mapsto & i\Lambda_\om (\bar\partial_{A}\partial_{A}-\partial_{A}\bar\partial_{A})\sigma.
\end{array}
$$
\begin{lemma}
\label{lem:exp4}
For $s_X=\pi^*\sigma_X\in\Gamma(X,\End(E))\subseteq\Gamma(X^\prime,\End(E))$, 
\begin{eqnarray*}
(\Delta_\varepsilon s_X)_X =\pi^*(\Delta_{X}\sigma_X)+\mathcal{O}(\varepsilon).
\end{eqnarray*}
\end{lemma}

\begin{proof}
There is $\phi\in\Gamma(X,\End(E))$ such that
$(\Delta_\varepsilon s_X)_X=\pi^*\phi$.
The element $\phi$ can be identified as the lowest order term in the asymptotic expansion in $\varepsilon$ of $(\Delta_\varepsilon\pi^*\sigma_X)_X$. However, we have at $x\in X'\setminus Z'$ :
\begin{eqnarray*}
\Delta_\varepsilon\pi^*\sigma_X=n\frac{D\pi^*\sigma_X\wedge (\omega+\varepsilon\lambda)^{n-1}}{(\omega+\varepsilon\lambda)^n} = n\pi^*\frac{D\sigma_X\wedge\omega^{n-1}}{\omega^n} +\mathcal{O}(\varepsilon)
\end{eqnarray*}
so we see that the lowest order term in the expansion of $(\Delta_\varepsilon\pi^*\sigma_X)_X$ is $\Delta_{X}\sigma_X$.
\end{proof}
Summarizing the above calculations, with respect to the decomposition $s=s_X+s_Z$ produced by (\ref{eq:goodcoordinates}), the operator $\Delta_\varepsilon$ takes the form 
\begin{equation}
 \label{eq:linearopmatrix}
\left(
\begin{array}{cc}
\Delta_X & 0\\
\mathcal{L} & \varepsilon^{-1}\Delta_\mathcal{V}
\end{array}
\right)
\end{equation}
plus higher order terms, for some second order operator $\cL$. 

\section{The approximate solutions}
\label{sec:approximate}
The goal of this section is to produce approximate solutions to the HYM equations. As in last section, expressions of the form $\cO(\ep^j)$ are to be understood as holding pointwise until Section \ref{sec:perturbation}.
As from now on the methods and proofs follow very closely \cite{SekTip}. We will only sketch most of the steps, giving details when major differences occur and precisely quoting results in \cite{SekTip} otherwise. Note that in our work, the parameter $\ep$ plays the role of $k^{-1}$ in \cite{SekTip}, and in the decomposition $\Gamma(X',\End E)=\Gamma(X,\End E)\oplus \Gamma_0(Z',\End E)$, the space $\Gamma(X, \End E)$ (resp. $\Gamma_0(Z',\End E)$) plays the role of $\Gamma(B,\End (E))$ (resp. of $\Gamma_0(X,\End E)$) in \cite{SekTip}.

\subsection{Fixing the complex deformation parameters}
\label{sec:complexdeformation}
We start from a semi-stable and sufficiently smooth holomorphic vector bundle $E$ on $(X,L)$, with $L=[\om]$. Denote by $\Gr(E)=\bigoplus_{i=1}^\ell\cG_i$ the associated polystable graded object, with stable components $\cG_i$. We also assume that $\Gr(E)$ has non-isomorphic stable quotients, and that the Jordan--Holder filtration with locally-free quotients is unique. We now explain the technical implications of these hypotheses. First, by \cite[Lemma 5.8]{SekTip}, $E$ is simple. The automorphism group $G:=\Aut(\Gr(E))$ is a reductive Lie group with Lie algebra $\g:=\aut(\Gr(E))$ and compact form $K\subset G$, with $\k:=\Lie(K)$. We have
$$
\g=\bigoplus_{i,j} H^0(X, \mathrm{Hom}(\cG_i,\cG_j).
$$
As the $\cG_i$'s are all stable, and non isomorphic, we deduce that
$$
\g=\bigoplus_{i=1}^\ell \C\cdot \Id_{\cG_i},
$$
and in particular $\g$ is abelian. The upshot for us is that elements in $\k$ that will obstruct the linear theory all live on the diagonal in the matrix block decomposition induced by the decomposition $\Gr(E)=\oplus_{i=1}^\ell \cG_i$.
We let $\delb_0$ be the Dolbeault operator of $\Gr(E)$. The Dolbeault operator $\delb_E$ on $E$ is given by
$$
\delb_E=\delb_0+\gamma
$$
where $\gamma \in\Om^{0,1}(X,\Gr(E)^*\otimes \Gr(E))$
can be written
$$
\gamma=\sum_{i<j} \gamma_{ij}
$$
for (possibly vanishing) $\gamma_{ij} \in\Om^{0,1}(X,\cG_j^* \otimes \cG_i)$. From the uniqueness hypothesis on the Jordan--H\"older filtration, we deduce that the terms $\gamma_{i,i+1}$, for $1\leq i \leq \ell-1$, are all non-zero, see \cite[Lemma 5.4]{SekTip}. This will enable us to use induction on $\ell$, the number of stable components, in the construction of the approximate solutions. 

Our starting point to produce HYM connections on $\pi^*E\to (X',\om_\ep)$ will be a product Hermite--Einstein metric $h= h_1\oplus \ldots \oplus  h_\ell$ on $\Gr(E)$ (we can fix one by polystability), and the associated HYM connection $A_0$ on $(\Gr(E),\delb_0)$ with contracted curvature form 
$$
\Lambda_{\om}iF_{A_0} = c_0 \cdot \Id.
$$ 
We then introduce the $L^2$ projection 
\begin{equation}
 \label{eq:L2projection}
 \begin{array}{cccc}
  \Pi_{i\k} :& \Gamma(X, \End_H (E,  h) ) & \to & i\k\\
             &      s            & \mapsto &\displaystyle \frac{1}{\Vol(\om)}\sum_{i=1}^\ell \frac{1}{\rank(\cG_i)} ( \int_X\trace_{\cG_i}(s)\:\om^n ) \: \Id_{\cG_i}
 \end{array}
\end{equation}
and the induced orthogonal decomposition:
\begin{equation*}
 \label{eq:L2splittingonB}
 \Gamma(X,\End_H(E,h)) =i\k \oplus \Gamma_{i\k^\perp}(X,\End_H(E, h))
\end{equation*}
with respect to the pairing 
$$
(s_1,s_2)\mapsto \int_X \trace(s_1\cdot s_2)\,\om^n.
$$
Denoting by $\Delta_{\Gr(E),0}$ the Laplace operator of the Chern connection of $(\delb_0,h)$ with respect to $\om$,  from Lemma \ref{lem:linop} and self-adjointness, we have 
\begin{lemma}
 \label{lem:Deltabaseinvertible}
The following operator is invertible : 
$$
\Delta_{\Gr(E),0} :\Gamma_{i\k^\perp}(X,\End_H(E,h)) \to \Gamma_{i\k^\perp}(X,\End_H(E, h)).
$$
\end{lemma}

We will now make use of the gauge action of $G$ to calibrate the complex deformation parameter $\gamma$ with the metric variation parameter $\ep$. We gauge fix $\gamma$ by imposing
\begin{equation}
\label{eq:gaugefixgamma}
\delb_0^*\gamma=0,
\end{equation}
where the adjoint is with respect to the K\"ahler structure $\om$. Elements 
$$g:=g_1\Id_{\cG_1}+\ldots,+g_\ell \Id_{\cG_\ell}\in G,$$
for $(g_i)\in(\C^*)^\ell$, act on $\delb_E$ and produce isomorphic holomorphic vector bundles in the following way :
\begin{equation}
\label{eq:explicitaction}
g\cdot \delb_E= \delb_0 + \sum_{i<j} g_ig_j^{-1}\gamma_{ij}.
\end{equation}
We assume from now on that for $1\leq i\leq \ell-1$, 
$$
\mu_{L_\ep}(\cF_i) \underset{\ep\to 0}{<}\mu_{L_\ep}(E). 
$$
We introduce $q_i$ to be the {\it discrepancy order} of $\cF_i$ for each $i\in[\![ 1, \ell-1]\!]$, that is the order $q_i\in\N$ such that $\mu_{L_\ep}(E)-\mu_{L_\ep}(\cF_i)=\nu_i \ep^{q_i}+\cO(\ep^{q_i+1}) $ for a constant $\nu_i>0$. The {\it discrepancy order} $q$ of $E$ will be $q=\max (q_i)_{1\leq i\leq \ell-1}$.  Define now 
$\underline{m}=(m_i)_{1\leq i \leq \ell-1}$ where for all $i$, $m_i$ satisfies\footnote{The formula differs from \cite{SekTip}, where it was $2m_i+1=q_i$, as here, when contracting upon $\Lambda_\ep$, we don't obtain an extra $\ep$-term (as was the case when contracting basic terms in \cite{SekTip})}
$$
2m_i = q_i.
$$ 
We then define ${\bf g}_{\underline{\upsilon}, \underline{m}}\in G$ by\footnote{Note that our convention for the gauge action is opposite to the one in \cite{SekTip}, which accounts for the negative powers in $\ep$ here.} 
\begin{equation*}
{\bf g}_{\underline{\upsilon}, \underline{m}} = 
\Id_{\cG_1} + \upsilon_1 \ep^{-m_1} \Id_{\cG_2} + \upsilon_1\upsilon_2 \ep^{-m_1-m_2}\Id_{\cG_3} +\ldots + (\Pi_{i=1}^{\ell-1}\upsilon_i) \ep^{-m_1 -\ldots-m_{\ell-1}} \Id_{\cG_\ell},
\end{equation*}
where the constants $\underline{\upsilon}=(\upsilon_i)_{1\leq i \leq \ell-1}\in(\R^*)^{\ell-1}$ will be determined soon. Denote by 
\begin{equation}
\label{eq:gammepdefinition}
\gamma_\ep:={\bf g}_{\underline{\upsilon}, \underline{m}}\cdot \gamma. 
\end{equation}
Then, the operators $\delb_E$ and $\delb_\ep=\delb_0+\gamma_\ep$ are gauge equivalent for any $\underline{\upsilon}$. We denote by $A_\ep$ the Chern connection on $(\pi^*E, \pi^*(\delb_\ep), \pi^*h)$ (and we will drop from now the $\pi^*$ to ease notations). We will consider projections onto various components of $\g$. For $\psi\in\g$ and $s=s_X+s_Z\in\Gamma(X',\End(E))=\Gamma(X,\End(E))\oplus \Gamma_0(Z',\End(E))$, we will denote by  $\Pi_{\langle \psi\rangle}(s)$ the orthogonal projection (as in Equation (\ref{eq:L2projection})) of $s_X$ onto the subspace spanned by $\psi$. Then, exactly as in \cite[Proposition 5.25]{SekTip}, using Chern-Weil theory together with the facts that the constants $\nu_i$ are all strictly positive and the $\gamma_{i, i+1}$ all non zero, we obtain :
\begin{proposition}
 \label{prop:fixingtheupsilons}
 There exist $\underline{\upsilon}=(\upsilon_1,\ldots,\upsilon_{\ell-1})\in (\R^*)^{\ell-1}$ and a positive constant $\cC$ such that for all $j\in [\![ 1, \ell-1 ]\!]$,
\begin{align*}
\Pi_{\langle \Id_{\cF_j}\rangle }  \Lambda_{\om_\ep} \left( i F_{A_\ep} \right) = \mathcal{C}\,\mu_{L_\ep} (E)\Id_{\cF_j} + O(\ep^{q_j+\frac{1}{2}}).
\end{align*}
\end{proposition}
In what follows, we assume that the constants $\underline{\upsilon}$ are fixed to satisfy the conclusion in Proposition \ref{prop:fixingtheupsilons}. The next step is to produce, for any order $p\geq q$, gauge perturbations $f_\ep\cdot A_\ep$ of $A_\ep$ satisfying
$$
\Lambda_{\om_\ep} \left( i F_{f_\ep\cdot A_\ep} \right) = \mathcal{C}\,\mu_{L_\ep} (E)\Id_E + O(\ep^{p}).
$$
\subsection{Inductive process and approximate solutions}
\label{sec:inductive}
To construct the approximate solutions to any desired order, we will perturb iteratively the connections $A_\ep$ by gauge transformations of the form $\exp(\ep^i s_i)$ and use the properties of the associated Laplace operators. We denote by $\Delta_\ep$ (resp. $\Delta_{\Gr(E),\ep}$) the associated Laplacian of $A_\ep$ (resp. of the Chern connection of $(\pi^*\delb_0,\pi^*h)$) with respect to $\omega_\ep$. We also set
$$
\Delta_X:=\Delta_{\Gr(E),0}
$$
the Laplacian of $(\delb_0,h)$ on $X$ and $\cL$ the second order operator associated to $(\delb_0,h)$ as in the matrix Expression (\ref{eq:linearopmatrix}).
\begin{proposition}
\label{prop:linleadingorder}
Under the above hypotheses, there are expansions
\begin{equation}
\label{eq:laplaceexpansion}
\Delta_\ep  = \Delta_{\Gr(E),\ep} + O(\ep),
\end{equation}
and for $s_X+s_Z\in\Gamma(X,\End E)\oplus \Gamma_0(Z',\End E)$, 
\begin{equation}
 \label{eq:laplacegradedexpansion}
 \Delta_{\Gr(E),\ep}(s_X+\ep s_Z)  =   \Delta_X (s_X) + \cL(s_X)+\Delta_\cV(s_Z) +\cO(\ep).
\end{equation}
 The same expansions also hold at a Chern connection on $\pi^* E \to X'$ coming from a complex structure $f_\ep \cdot \pi^*(\delb_0+\gamma_\ep)$ provided $f_\ep = \Id_E+ s_\ep$ for some $s_\ep \in \Gamma \left( X', \End  \pi^* E \right)$  whose $X$-component $s_{X,\ep}$ satisfies $s_{X,\ep} = O(\ep)$ and whose $Z'$-component $s_{Z, \ep}$ satisfies $s_{Z,\ep} = O(\ep^2)$.
\end{proposition}
\begin{proof}
 For the proof of (\ref{eq:laplaceexpansion}), we use the formula for the change of curvature induced by the change in complex structure $\delb_0\mapsto \delb_0+\gamma_\ep$ : 
 \begin{equation}
  \label{eq:curvaturefromchangecomplexstructure}
  F_{A_\ep}=F_{A}+d_{A_0}(\gamma_\ep-\gamma_\ep^*)+(\gamma_\ep-\gamma_\ep^*)\wedge (\gamma_\ep-\gamma_\ep^*).
 \end{equation}
 Upon contraction with $\om_\ep$, we obtain
 $$
 \Lambda_\ep (iF_{A_\ep})=\Lambda_\ep (iF_{A_0})+\Lambda_\ep \left(\del_0\gamma_\ep-\delb_0\gamma_\ep^*\right)-2\Lambda_\ep \left(\gamma_\ep\wedge\gamma_\ep^*\right).
 $$
Now, by choice of $g_{\underline{\upsilon},\underline{m}}$ and definition of $\gamma_\ep$ in Equation (\ref{eq:gammepdefinition}), we see that 
$$
\gamma_\ep=\cO(\ep^{\min (m_i)}).
$$
By definition, $m_i=\frac{q_i}{2}$. By Formula (\ref{eq:formulaAchim}), for all $i\in [\![1, \ell-1]\!]$, we have $q_i\geq n-m$, and as we are blowing-up a submanifold, we must have $m\leq n-2$. Then, $q_i \geq 2$, and thus $m_i\geq 1$ for all $i$.
Then,
$$
\gamma_\ep=\cO(\ep),
$$
which gives
$$
 \Lambda_\ep (iF_{A_\ep})=\Lambda_\ep (iF_{A_0})+\cO(\ep).
$$
Note that all the terms in $ \left(\del_0\gamma_\ep-\delb_0\gamma_\ep^*\right)-2\left(\gamma_\ep\wedge\gamma_\ep^*\right)$ are pulled back terms, and arguing as in Section \ref{sec:decomposeLaplace}, we obtain that this estimate is preserved for first order variations, which gives 
$$
\Delta_\ep  = \Delta_{\Gr(E),\ep} + O(\ep).
$$
The proof of (\ref{eq:laplacegradedexpansion}) is a direct consequence of the various lemmas in Section \ref{sec:decomposeLaplace}. Finally, we consider
the perturbed connections under gauge transformations $f_\ep$ as described in the statement of the lemma. The expansion (\ref{eq:laplaceexpansion}) is trivially preserved for such connections, as $f_\ep=\Id_E+\cO(\ep)$. Then, the $X$-part of   $ \partial_{A_\ep} \bar \partial_{A_\ep}  - \bar \partial_{A_\ep}  \partial_{A_\ep}$ changes at order $\ep$ while its $Z'$-part changes at order $\ep^2$, which is enough to compensate for the $\ep^{-1}$-contribution that comes upon restriction to $Z'$ and contraction, as in the proof of Lemma \ref{lem:exp1}. This concludes the proof.
\end{proof}
When producing the approximate HYM connections, we will use the mapping properties of $\Delta_X$ and $\Delta_\cV$. This will allow us to remove all error terms modulo elements in the cokernel of $\Delta_X=\Delta_{\Gr(E),0}$, which is $\g$. By Chern-Weil theory, those remaining terms will be controlled by the expansions of $\mu_{L_\ep}(\cF_i)$, and it is important in the argument that upon removing the other errors, they remain fixed. This is the content of the following lemma.
\begin{lemma}
 \label{lem:dontchangethetraces!}
 Let $s\in \Gamma(\End(\cG_i))$ and $j\in[\![ 1, \ell ]\!]$. Then
 $$
 \Pi_{\langle\Id_{\cF_j}\rangle} \Delta_\ep( s) = \cO(\ep^{q_j})
 $$
 and
 $$
 \Pi_{\langle\Id_{\cF_j}\rangle} \Lambda_\ep\left((\partial_\ep s-\bar \partial_\ep s)\wedge (\partial_\ep s-\bar \partial_\ep s) \right)=\cO(\ep^{q_j}).
 $$
 Let $\sigma\in \Gamma(\Hom(\cG_p,\cG_i))$, with $i< p$, and $j\in[\![ 1, \ell ]\!]$. Then
 $$
 \Pi_{\langle\Id_{\cF_j}\rangle} \Delta_\ep(\ep^{m_i+\ldots+m_{p-1}+1}\sigma) =\cO(\ep^{q_j+1}).
 $$
and
  $$
 \Pi_{\langle\Id_{\cF_j}\rangle} \Lambda_{\ep}\left(\ep^{2(m_i+\ldots+m_{p-1}+1)}\left((\partial_\ep \sigma-\bar \partial_\ep \sigma)\wedge (\partial_\ep \sigma-\bar \partial_\ep \sigma) \right)\right)=\cO(\ep^{q_j+1}),
 $$
and similarily for $p<i$.
\end{lemma}
\begin{proof}
 Note that we only need to consider the $\Gamma(X,\End E)$ components of the sections in the splitting (\ref{eq:goodcoordinates}), as we are only interested in the projection onto $\Id_{\cF_j}$. Then, the proof follows exactly as in \cite[Lemma 5.20 and Lemma 5.21]{SekTip}. 
\end{proof}
The only errors that will then remain to remove are those in $\g$. This is where we will use the following lemma, whose proof is a direct adaptation of \cite[proof of Lemma 5.23]{SekTip}.
\begin{lemma}
 \label{lem:linearisationIdis}
  For all $j\in [\![ 1, \ell-1]\!]$, there is a negative constant $a_{j,j+1}$ such that
\begin{equation}
 \label{eq:deltaactionIdisonautg}
\Pi_{\g}\Delta_\ep(\Id_{\cF_j}) =  a_{j,j+1} \ep^{q_j}\Id_{j,j+1}+\cO(\ep^{q_j+1}),
\end{equation}
where for $(p,l)\in [\![1,\ell]\!]^2$:
 $$
 \Id_{pl}=\frac{1}{\rk\,\cG_p}\Id_{\cG_p}-\frac{1}{\rk\,\cG_{l}}\Id_{\cG_{l}}.
 $$
\end{lemma}
Note that as in Proposition \ref{prop:linleadingorder}, the conclusions of Lemma \ref{lem:dontchangethetraces!} and \ref{lem:linearisationIdis} also hold for perturbed connections of the form $f_\ep\cdot A_\ep$ for $f_\ep=\Id_E+s_{\ep,X}+s_{\ep,Z}$ with $s_{\ep,X}=\cO(\ep)$ and $s_{\ep,Z}=\cO(\ep^2)$. We are now ready to prove the main proposition of this section.
\begin{proposition}
 \label{prop:approximate}
 Let $p\geq q$. Then, there exists gauge transformations $f_\ep^p=\Id_E+s_{\ep,X}^p+s_{\ep,Z}^p$ with $s_{\ep,X}^p=\cO(\ep)$ and $s_{\ep,Z}^p=\cO(\ep^2)$, and constants $c_\ep^p$, such that for all $0<\ep\ll 1$, we have 
 $$
 \Lambda_\ep\left(F_{f_\ep^p\cdot A_\ep}\right)=c_\ep^p \Id_E + \cO(\ep^p).
 $$
\end{proposition}
\begin{proof}
 The proof follows the strategy developped in \cite[Section 5.2 and Section 5.3.2]{SekTip}, so we refer to those relevant sections and will only sketch the argument.  We can expand in $\ep$ each term on the right hand side of the following identity :
  $$
 \Lambda_\ep (iF_{A_\ep})=\Lambda_\ep (iF_{A_0})+\Lambda_\ep \left(\del_0\gamma_\ep-\delb_0\gamma_\ep^*\right)-2\Lambda_\ep \left(\gamma_\ep\wedge\gamma_\ep^*\right).
 $$
 As $A_0$ is HYM, we have 
 $$
 \Lambda_\ep (iF_{A_0})=c\Id_E + \cO(\ep).
 $$
By the gauge fixing condition $\delb_0^*\gamma=\Lambda_0\del_0 \gamma=0$,  the following off-diagonal terms have entries 
 $$
 (\Lambda_\ep \left(\del_0\gamma_\ep-\delb_0\gamma_\ep^*\right))_{ip}=\cO(\ep^{m_i+\ldots+m_{p-1}+1}),
 $$
 while the remaining terms will contribute to higher orders. Following \cite[Proposition 4.9 and Section 5.2]{SekTip}, we start by perturbing inductively $A_\ep$ to remove all errors that live in $\Gamma_{i\k^\perp}(X,\End_H(E, h))\oplus\Gamma_0(Z', \End_H(E,h))$. This can be achieved by considering perturbations $\exp(\ep^i(s_X+\ep s_Z)\cdot A_\ep)$ thanks to Proposition \ref{prop:linleadingorder}, together with Lemmas \ref{lem:verticallaplacianinvertible} and \ref{lem:Deltabaseinvertible}. What remains are errors in $i\k$. By Lemma  \ref{lem:dontchangethetraces!}, and from the formula 
 $$
 \Lambda F_{f\cdot A}=\Lambda F_A + \Delta_A(f) +\Lambda \left((\del_A f - \delb_A f)\wedge (\del_A f - \delb_A f)\right),
 $$
 the previous perturbations won't affect any of the $\Id_{\cF_j}$ projections of the contracted curvatures. By proposition \ref{prop:fixingtheupsilons}, we then obtain a connection whose contracted curvature satisfies, for all $j\in [\![ 1, \ell-1 ]\!]$,
 $$
 \Lambda_{\om_\ep} \left( i F_{A_\ep} \right) = \mathcal{C}\,\mu_{L_\ep} (E)\Id_{\cF_j} + O(\ep^{q_j+\frac{1}{2}}).
$$
The end of the proof is then done by  induction on $\ell$ the number of stable components of $\Gr(E)$. As discussed in \cite[Section 5.3.2]{SekTip}, using Lemma \ref{lem:linearisationIdis}, we can remove errors in $\g$ beyond the orders $q_j$, by induction on $\ell$, and eventually obtain the result.
 \end{proof}

\section{The perturbation argument}
\label{sec:perturbation}
We keep notations from the last section. We will now prove Theorem \ref{theo:intro}, and its corollaries. Again, we follow closely \cite{SekTip}. The first step is to obtain an upper bound for the operator norm of the inverse of $\Delta_\ep$. This is where the argument from \cite{SekTip} has to be adapted, given the different geometric settup that we address. Then, based on this bound and on the construction of approximate solutions, a quantitative version of the implicit function theorem is used in Section \ref{sec:perturbing}. 

For a Riemannian metric $g$ on $X'$, let $L^2_{d} (g)$ denote the Sobolev space $W^{d,2}(X', g)$ of order $2$ and $d$ derivatives, with respect to the metric $g$. If $g$ is K\"ahler with K\"ahler form $\eta$, we may write $L^2_d (\eta)$ instead of $L^2_d (g)$. When $d=0$, we omit the subscript. For Sobolev spaces associated to $\pi^*E$ or $\End \pi^*E$, we use the metric $h$ (and the metric it induces on $\End \pi^*E$) on the bundle.

\subsection{Estimating the linearisation}
\label{sec:estimate}
Our main goal in this section is to prove:
\begin{proposition}
\label{prop:lindiagsbound}                                                    
There exists $C> 0$ such that for all $s\in \Gamma(X',\End E)$ whose trace is average zero with respect to $\om_\ep$,
\begin{align}
\label{eqn:laplacianbound}
\| \Delta_\ep (s) \|_{L^2_{d}(\omega_\ep)} \geq C \ep^q \| s\|_{L^2_{d+2}(\omega_\ep)}.
\end{align}
The same estimate also holds for the Laplacian of the perturbed connections built in Proposition \ref{prop:approximate}.
\end{proposition}
This will be done by steps. We first work out the case of a single stable component $\cG:=\cG_i\subset \Gr(E)$.
\begin{lemma}
 \label{lem:singlecomponent}
 There exists $C> 0$ such that for all $s\in \Gamma(X',\End \cG)$ that is $L^2(\om_\ep)$-orthogonal to $\Id_\cG$,
\begin{align}
\label{eqn:laplacianboundsingle}
\| \Delta_\ep (s) \|_{L^2_{d}(\omega_\ep)} \geq C \ep \| s\|_{L^2_{d+2}(\omega_\ep)}.
\end{align}
\end{lemma}

The proof of Lemma \ref{lem:singlecomponent} will require to compare various $L^2$-norms on $X'$, $X$ and $Z'$. We gather those comparisons in the following lemmas.
\begin{lemma}
 \label{lem:intermediatecomparisonnormspulledback}
There are positive constants $C$ and $C'$ independent on $\ep$ such that for any $\sigma\in \Gamma(X,\End \cG)$ (or $\sigma \in \Om^1(X,\End \cG)$), we have 
  \begin{equation}
   \label{eq:comparisonbasics}
  C \| \pi^*\sigma \|_{L^2(X',\omega_\ep)}\leq \|  \sigma \|_{L^2(X,\omega)} \leq C' \|\pi^*\sigma \|_{L^2(X',\omega_\ep)}.
  \end{equation}
\end{lemma}

\begin{proof}
  The result follows from the facts that the $2$-forms $\om_\ep$ vary in a bounded family ($\ep\in[0,\ep_0]$), and that  a pulled back section is constant on the fibers of $\pi : Z' \to Z$.
\end{proof}

\begin{lemma}
 \label{lem:intermediatecomparisonnormsextended}
 There are positive constants $C$ and $C'$ independent on $\ep$ such that for any $\sigma\in \Gamma_0(Z',\End \cG)$ (or $\sigma \in \Om^1(Z',\End \cG)$), we have 
  \begin{equation}
   \label{eq:comparisonextended}
  C \| \iota_*\sigma \|_{L^2(X',\omega_\ep)}^2 \leq    \|  \sigma \|_{L^2(Z',\omega_\ep)}^2 \leq C'\| \iota_*\sigma \|_{L^2(X',\omega_\ep)}^2 .
  \end{equation}
\end{lemma}

\begin{proof}
 This can be done by a local and then patching argument. We may fix $W\subset X'$ an open set with $Z'\subset W$, and such that  $\iota_*\sigma$ vanishes away from $W$. Then, 
$$
\|  \iota_*\sigma \|_{L^2(X',\omega_\ep)}^2=\int_W \|  \iota_*\sigma \|_\ep^2 \: \om_\ep^n.
$$
We may also consider an open finite covering of $W$ by sets $U$ of the form $U\simeq B(0,1)\times V$ where the sets $V\subset Z'$ cover $Z'$ and such that $\iota_*\sigma$ is given locally by $\rho\cdot \sigma$ for a bump function $\rho$ on $B(0,1)$. Then, from the facts that $\ep$ varies in a compact family, and that for any $v\in V$, the metrics $(\om_\ep(b,v))_{b\in B(0,1)}$ on $B(0,1)\times \lbrace v \rbrace$ are mutually bounded, by applying Fubini's theorem, we obtain bounds :
$$
C\int_U \|  \iota_*\sigma \|_\ep^2 \: \om_\ep^n \leq \int_V \|  \sigma \|_\ep^2 \: \om_\ep^{n-1} \leq C' \int_U \|  \iota_*\sigma \|_\ep^2 \: \om_\ep^n.
$$
A patching argument then provides the result.
\end{proof}

\begin{lemma}
 \label{lem:comparisonmixedterms}
 There exists a positive constant $C$ independent on $\ep$ such that for any $s_X\in \Gamma(X,\End \cG)$ and $s_Z\in \Gamma_0(Z',\End \cG)$, we have
 $$
 \vert \langle s_X, s_Z \rangle_{L^2(\om_\ep)}\vert\leq C\ep^3  \vert\vert s_X\vert\vert_{L^2(\om_\ep)}\,\vert\vert s_Z\vert\vert_{L^2(\om_\ep)} .
  $$
\end{lemma}

\begin{proof}
We keep the notations from the proof of Lemma \ref{lem:intermediatecomparisonnormsextended}.
  Note that we can restrict to the open set $W$ as $s_Z$ vanishes away from $W$. Then, locally, we estimate 
  $$
  \int_U \trace(s_X\cdot s_Z) \: \om_\ep^n.
  $$
  Take coordinates $(b,f,v)\in B(0,1)\times F \times V_Z$ so that $U\simeq B(0,1)\times F \times V_Z$ where this time $F$ denote fibers of $\pi : Z' \to Z$, $V_Z\subset Z$ is an open set in $Z$ and $B(0,1)$ stands for normal coordinates. We may assume that $s_X$ is independent on $f\in F$ : 
  $$
  s_X=s_X(b, v)
  $$
  and $s_Z$ is of the form 
  $$
  s_Z=\rho(b) s_Z(f,w)
  $$
  for $\rho$ some cut off function used to produce the extension operator. Then, we compute with Fubini's :
  $$
  \begin{array}{ccc}
    \displaystyle \int_U \trace(s_X\cdot s_Z) \: \om_\ep^n & = &  \displaystyle \int_{B\times F\times V_Z}\rho(b)\,\trace(s_X(b,w)\cdot s_Z(f,w)) \,\om_\ep^n\\
     & = & \displaystyle \int_{B\times V_Z}\rho(b)\int_F \trace(s_X(b,w)\cdot s_Z(f,w))\, \om_\ep^n.
  \end{array}
  $$
  But to higher order in $\ep$, 
  $$
  \displaystyle \int_F \trace(s_X(b,w)\cdot s_Z(f,w))\, \om_\ep^n =$$
  $$
   \displaystyle \ep^{n-m-1}\om^{m+1}\int_F \trace(s_X(b,w)\cdot s_Z(f,w))\, \lambda^{n-m-1}
    +\cO(\ep^{n-m})
  $$
  and as $s_Z$ is of average zero w.r.t. $\lambda^{n-m-1}$, we get
   $$
  \int_F \trace(s_X(b,w)\cdot s_Z(f,w)) \om_\ep^n=\cO(\ep^{n-m}),
  $$
  which is $\cO(\ep^3)$ as $\codim(Z)=n-m\geq 3$. The result then follows from Cauchy-Schwarz inequality, for a constant $C$ that depends on $\om, \lambda$ and $\ep_0$.
  \end{proof}

\begin{proof}[Proof of Lemma \ref{lem:singlecomponent}]
 The positive constants $C, C', C_i$ used  in the proof might vary at several stages. We follow the strategy from \cite[Section 4.2]{SekTip}. The key to obtain this estimate is the analogue of \cite[Lemma 4.12]{SekTip}, as the rest of the argument follows as in \cite[Section 4.2]{SekTip}. So our goal is to obtain the following Poincar\'e type inequality:
 there exists $C >0$ such that for any $\ep\in [0, \ep_0]$, and all $s\in \Gamma(X',\End \cG)$ that is $L^2(\om_\ep)$-orthogonal to $\Id_\cG$, we have :
 $$
 \| d_{A_\ep} (s) \|_{L^2(\omega_\ep)}^2 \geq C \ep \| s\|_{L^2(\omega_\ep)}^2.
 $$
 As $A_\ep=A_0+\cO(\ep)$, it is actually enough to obtain this Poincar\'e inequality for $A_0$:
  $$
 \| d_{A_0} (s) \|_{L^2(\omega_\ep)}^2 \geq C \ep \| s\|_{L^2(\omega_\ep)}^2.
 $$
 We will then prove this in three steps. First, for sections $s\in \Gamma(X,\End \cG)$, then for sections $s\in \Gamma_0(Z', \End \cG)$ and last for sums of such sections.
 
  {\it Step 1 :}   For $\sigma\in\Gamma(X,\End \cG)$ and $s=\pi^*\sigma\in\Gamma(X^\prime,\End \cG)$, let 
\begin{eqnarray*}
\alpha &=& \frac{1}{\rank\,{\cG}}\,\frac{1}{\mathrm{vol}(X,\omega)}\int_X \trace(\sigma)\, \omega^n,\\
\alpha^\prime &=& \frac{1}{\rank\,{\cG}}\,\frac{1}{\mathrm{vol}(X^\prime,\omega_\varepsilon)}\int_{X^\prime} \trace(s)\, \omega_\varepsilon^n.
\end{eqnarray*}
Then, $s$ is orthogonal to $\Id$ on $X^\prime$ if and only if $\alpha^\prime=0$. Moreover, 
\begin{eqnarray*}
\alpha^\prime-\alpha =\cO(\varepsilon)\int_X \trace(\sigma)\,\omega^n.
\end{eqnarray*} 
Then, noting that $d_{A_0}\sigma=d_{A_0}(\sigma-\alpha \Id)$ as $\Id$ is parallel, we have
\begin{eqnarray*}
\|d_{A_0}(\pi^*\sigma)\|_{L^2(X',\omega_\varepsilon) }^2 &\geq & C_1\|\pi^*d_{A_0}\sigma\|_{L^2(X',\omega_\varepsilon)}^2,\\
&\geq & C_2\| d_{A_0}\sigma\|_{L^2(X,\omega)}^2,\\
&\geq& C_3 \|\sigma-\alpha\Id\|_{L^2(X,\omega)}^2,\\
&\geq & C_4 \|\sigma\|_{L^2(X,\omega)}^2,\ \ \\
&\geq & C_5\| s\|_{L^2(X',\omega_\varepsilon)}^2,
\end{eqnarray*}
where the second and fifth inequalities follow from Lemma \ref{lem:intermediatecomparisonnormspulledback}, the third one from Poincar\'e inequality for $A_0$ on $(X,\om)$, and the fourth inequality follows since $\alpha$ is  $\mathcal{O}(\varepsilon\|\sigma\|_{L^2(\omega)})$.
  
{\it Step 2 :} We next consider $s=\iota_*\sigma\in \Gamma_0(Z', \End \cG)$.  We set this time $\alpha^{\prime\prime}$ to be the constant given by the $L^2(Z',\om_\ep)$-orthogonal projection of $\sigma$ onto $\Id_E$ :
\begin{eqnarray*}  
\alpha^{\prime\prime} &=& \frac{1}{\mathrm{rank}\,\mathcal{G}\, \mathrm{vol}(Z^\prime,\omega_\varepsilon)}\int_{Z^\prime} \trace(\sigma)\,\omega^{n-1}_\varepsilon.  
\end{eqnarray*}  
  In general, for a given $\sigma\in\Gamma(Z^\prime,\End \cG)$, $\alpha^{\prime\prime}=\mathcal{O}(1)$ with respect to $\varepsilon$. However, by definition of the subspace $\Gamma_0(Z^\prime,\End \cG)\subset \Gamma(Z^\prime,\End \cG) $, since $\sigma\in\Gamma_0(Z^\prime,\End \cG)$ is orthogonal to the identity on all fibres of $\pi: Z'\to Z$, an argument similar to that of Lemma \ref{lem:comparisonmixedterms}, using the same system of coordinates $(b,f,v)\in B(0,1)\times F\times V_Z$, shows that $\alpha^{\prime\prime}=\mathcal{O}(\varepsilon\|\sigma\|_{L^2(Z',\omega_\ep)})$. Then we find positive constants $C_i$ such that
\begin{eqnarray*}
   \| d_{A_0} (\iota_*\sigma) \|_{L^2(\omega_\ep)}^2 & \geq &  C_1 \| \iota_*(d_{A_0} \sigma) \|_{L^2(X',\omega_\ep)}^2\\
    & \geq & C_2 \| d_{A_0} \sigma \|_{L^2(Z',\omega_\ep)}^2\\
     & \geq & \ep\: C_3 \|  \sigma -\alpha^{\prime\prime}\mathrm{Id} \|_{L^2(Z',\omega_\ep)}^2\\
     & \geq & \ep\: C_4 \| \sigma \|_{L^2(Z',\omega_\ep)}^2,\\
      & \geq & \ep\: C_5 \|  \iota_*\sigma \|_{L^2(X',\omega_\ep)}^2,
\end{eqnarray*}
  where this time the first inequality follows from construction of the operator $\iota_*$, the second and last inequalities come from Lemma \ref{lem:intermediatecomparisonnormsextended}, the third one follows from the corresponding \cite[Lemma 4.12]{SekTip}, and the fourth one from the fact that $\alpha^{\prime\prime}=\mathcal{O}(\varepsilon\|\sigma\|_{L^2(Z',\omega_\ep)})$.

  {\it Step 3 :}  We need now to obtain similar estimates for sections $s=s_X+s_Z$. Note however that the splitting of (\ref{eq:goodcoordinates}) is {\it not} orthogonal with respect to the $L^2(\om_\ep)$ inner product.  What remains is to estimate 
  $$
  \langle d_{A_0} s_X, d_{A_0} s_Z \rangle_{L^2(\om_\ep)}=\langle \Delta_{A_0,\ep}(s_X),  s_Z \rangle_{L^2(\om_\ep)}=\langle s_X,  \Delta_{A_0,\ep}(s_Z) \rangle_{L^2(\om_\ep)},$$
  for $\Delta_{A_0,\ep}$ the Laplacian of $A_0$ with respect to $\om_\ep$. 
   We use the expansion of $\Delta_{\ep}$ from Section \ref{sec:decomposeLaplace}. From Lemma \ref{lem:exp3}, Lemma \ref{lem:exp4} and formula (\ref{eq:linearopmatrix}) :
  $$
  \langle \Delta_{A_0,\ep}(s_X),  s_Z \rangle_{L^2(\om_\ep)}=\langle \Delta_X s_X, s_Z \rangle_{L^2(\om_\ep)}+  \langle \cL( s_X), s_Z \rangle_{L^2(\om_\ep)}+\cO(\ep)
  $$
  Then, from Lemma \ref{lem:comparisonmixedterms}, 
  $$
   \vert \langle \Delta_X s_X, s_Z \rangle_{L^2(\om_\ep)} \,\vert  \leq   C_1\, \ep^3 \vert\vert \Delta_X s_X \vert\vert_{L^2(\om_\ep)}\,\vert\vert s_Z\vert\vert_{L^2(\om_\ep)}.
   $$
  Setting $s_X=\pi^*\sigma$, and using Lemma \ref{lem:intermediatecomparisonnormspulledback}, together with the continuity of the Green operator of $\Delta_X$, we obtain
 \begin{eqnarray*}
  \vert \langle \Delta_X s_X, s_Z \rangle_{L^2(\om_\ep)} \,\vert    & \leq & C_2\,\ep^3 \vert\vert \Delta_X \sigma \vert\vert_{L^2(\om)}\,\vert\vert s_Z\vert\vert_{L^2(\om_\ep)},\\
      & \leq & C_3\,\ep^3 \vert\vert  \sigma \vert\vert_{L^2(\om)}\,\vert\vert s_Z\vert\vert_{L^2(\om_\ep)},\\
        & \leq & C_4\,\ep^3 \vert\vert s_X \vert\vert_{L^2(\om_\ep)}\,\vert\vert s_Z\vert\vert_{L^2(\om_\ep)},
 \end{eqnarray*}
 where again we dealt with the $L^2(X,\omega)$-projection of $\sigma$ onto $\Id_\cG$ as in Step 1.
For the term $  \langle \cL (s_X), s_Z \rangle_{L^2(\om_\ep)}$, we use local coordinates $(b,f,v)\in B\times F \times V_Z$ as in the proof of Lemma \ref{lem:comparisonmixedterms}. Then, the expression of $\cL$ can be locally written 
\begin{eqnarray*}
  \cL(s_X) & = & \iota_*\circ p_0\circ \iota^*\left((m+1)\frac{Ds_X\wedge\omega^m\wedge\lambda^{n-m-1}}{\omega^{m+1}\wedge\lambda^{n-m-1}}\right)\\
 & = & (m+1)\frac{D_b s_X\wedge\omega^m\wedge\lambda^{n-m-1}}{\omega^{m+1}\wedge\lambda^{n-m-1}}
\end{eqnarray*}
  where $D_b$ stands for derivatives in the $B(0,1)$-direction. We then compute the local contribution on $U$ of the term $\langle \cL (s_X), s_Z \rangle_{L^2(\om_\ep)}$ :
  $$
  \int_U \trace(\cL (s_X)\cdot s_Z) \: \om_\ep^n=\int_{B\times F\times V_Z}\rho(b)\,\trace(\cL(s_X(b,w))\cdot s_Z(f,w)) \,\om_\ep^n,
  $$
  and by Fubini's, Stokes theorem, and integration by parts, the relevant higher order term in $\ep$ is 
  $$
  \int_U D_b(\rho(b))\, \trace (s_X \cdot s_Z) \om_\ep^n,
  $$
  which is $\cO(\ep^3\,\vert\vert s_X\vert\vert_{L^2(\om_\ep)}\,\vert\vert s_Z\vert\vert_{L^2(\om_\ep)})$ by a similar argument as in Lemma \ref{lem:comparisonmixedterms}. Gathering those estimates, we see that the mixed terms satisfy 
  $$\langle d_{A_0} s_X, d_{A_0} s_Z \rangle_{L^2(\om_\ep)}=\langle \Delta_{A_0,\ep}(s_X),  s_Z \rangle_{L^2(\om_\ep)}=\cO(\ep^3\,\vert\vert s_X\vert\vert_{L^2(\om_\ep)}\,\vert\vert s_Z\vert\vert_{L^2(\om_\ep)}),$$
  and together with steps 1 and 2, using Lemma \ref{lem:comparisonmixedterms} again, we obtain 
 $$
 \| d_{A_0} (s_X+s_Z) \|_{L^2(\omega_\ep)}^2 \geq C \ep \| s_X+s_Z\|_{L^2(\omega_\ep)}^2.
 $$
 
Then, from this Poincar\'e in\'equality, the result follows as in \cite[Section 4.2]{SekTip}, using an analogous uniform Schauder estimate as in \cite[Proposition 4.13]{SekTip}. The latter estimate can be obtained by patching local Schauder estimates (as in \cite[Lemma 4.14]{SekTip}) that are easily derived away from the exceptional divisor, and can be obtained around a point in the exceptional divisor by adapting \cite[Proof of Lemma 4.14]{SekTip} using local coordinates as in Lemma \ref{lem:comparisonmixedterms}.
\end{proof}

Once this bound settled, we can deal with the general case.

\begin{proof}[Proof of Proposition \ref{prop:lindiagsbound}]
 The proof for the estimate (\ref{eqn:laplacianbound}) can be established exactly as in \cite[Proposition 5.27]{SekTip}, following \cite[Lemma 5.28, Lemma 5.29]{SekTip}.
 First, to obtain the estimate for elements in $\g^\perp$,
  one uses the full expansion in the linearisation :
 \begin{equation}
 \label{eq:fullexpansionlaplace}
\begin{array}{ccc}
 \Delta_\ep & = & \Delta_{\Gr(E),\ep}\\
          & + &i \Lambda_{\ep} \left(\delb_0 ([\gamma_{\ep}^*,\cdot ])-[\gamma_{\ep}^*,\delb_0 \cdot ]\right)\\
          & + & i \Lambda_{\ep} \left(\partial_0([\gamma_{\ep},\cdot])-[\gamma_{\ep},\partial_0\cdot ] \right) \\
& + & i \Lambda_{\ep} ( [\gamma_{\ep}, [\gamma_{\ep}^*, \cdot ] ] - [\gamma_{\ep}^*, [\gamma_{\ep}, \cdot ] ] ).
 \end{array}
\end{equation}
As $\gamma_\ep$ is a pulled back term, arguing as in Section \ref{sec:decomposeLaplace}, we see that the $X$ and $Z'$ contributions of the operator $\Lambda_{\ep} \left(\partial_0([\gamma_{\ep},\cdot])\right)$ in the splitting (\ref{eq:goodcoordinates}) will be of the same order as $\gamma_\ep$. As argued before, from the choice of $g_{\underline{\upsilon},\underline{m}}$ and definition of $\gamma_\ep$ in Equation (\ref{eq:gammepdefinition}), we see that 
$$
(\gamma_\ep)_{ip}=\cO(\ep^{m_i}).
$$
Now, $m_i=\frac{q_i}{2}\geq \frac{\mathrm{codim}(Z)}{2}$ by Formula (\ref{eq:formulaAchim}). Arguing similarily for the other terms in (\ref{eq:fullexpansionlaplace}), we see that the contributions from $\Delta_\ep -  \Delta_{\Gr(E),\ep}$ will all come at order at least $\ep^\frac{\mathrm{codim}(Z)}{2}$, and will be absorbed by the estimate in Lemma \ref{lem:singlecomponent}, because we assumed $\codim(Z)\geq 3$. Arguing as in \cite[Lemma 5.28]{SekTip}, we obtain the bound :
\begin{align*}
\| \Delta_\ep (s) \|_{L^2_{d}(\omega_\ep)} \geq C \ep \| s\|_{L^2_{d+2}(\omega_\ep)},
\end{align*}
for $s \in \g^\perp$. Then, from here, the proof for the estimate for sections in $\g$, or for sums of sections, follows as in \cite[proof of Lemma 5.29 and proof of Proposition 5.27]{SekTip}. Finally, the result for perturbed connections follows as in \cite[Proposition 5.30]{SekTip}.
\end{proof}

\subsection{Perturbing to solutions}
\label{sec:perturbing}

To conclude the proof of Theorem \ref{theo:intro}, we refer to \cite[Section 4.2 and Section 5.4.2]{SekTip}. It relies on a quantitative version of the implicit function theorem, applied to the operator
$$
\begin{array}{ccccc}
 \Psi_{p,\ep} : &L^2_{d+2}(X',\om_\ep)\times \R & \to & L^2_{d}(X',\om_\ep)\\
 &(s,\alpha) & \mapsto & i\Lambda_{\om_\ep}(F_{\exp(s)\cdot f_\ep^p\cdot A_\ep})-\alpha \Id_E,
\end{array}
$$
where $f_\ep^p\cdot A_\ep$ are the connections built in Proposition \ref{prop:approximate}. Proposition \ref{prop:approximate} is the analogue of \cite[Proposition 4.9 and Section 5.3.2]{SekTip} where approximate solutions to $\Psi_{p,\ep}=0$ are constructed, while Proposition \ref{prop:lindiagsbound} plays the role of \cite[Propositions 5.27 and 5.30]{SekTip} and provides the required estimate on the linearisation of $\Psi_{p,\ep}$, for $p$ large enough, at the approximate solutions. Then, the implicit function theorem, as stated in  \cite[Theorem 4.10]{SekTip}, enables to conclude the existence of zeros for $\Psi_{p,\ep}$, for $p$ large and $\ep$ small, which ends the proof of Theorem \ref{theo:intro}. 

\subsection{Proof of the corollaries}
\label{sec:corol}
We comment now on the various corollaries stated in the introduction. First, Corollary \ref{cor:stablecase} is a direct application of Theorem \ref{theo:intro}, where $E=\Gr(E)$ as a single stable component. Corollary \ref{cor:restrictionnumericalcriterion} also follows directly, using Formula (\ref{eq:formulaAchim}). What remains is to show Corollary \ref{cor:semistablecase}. The only remaing case to study is when for all $i\in[\![1,\ell-1]\!]$,  $\mu_{L_\ep}(\cF_i)\underset{\ep\to 0}{\leq}\mu_{L_\ep}(E)$, with at least one equality. This case can be dealt with exactly as for \cite[Corollary 5.2]{SekTip}, proved in \cite[Section 5.5]{SekTip}.

\bibliographystyle{plain}	
 \bibliography{ClaSekTip}

\end{document}